\documentclass[12pt]{amsart}

\usepackage{amsmath}
\usepackage{amssymb}
\usepackage{amsfonts}
\usepackage{amsthm}
\usepackage{enumerate}
\usepackage{hyperref}
\usepackage{color}

\newcommand{\KK}{\mathbb{K}}
\newcommand{\NN}{\mathbb{N}}
\newcommand{\set}[1]{\{#1\}}
\newcommand{\with}{\,:\,}

\DeclareMathOperator{\sdepth}{sdepth}
\DeclareMathOperator{\depth}{depth}

\theoremstyle{plain}
\newtheorem{thm}{Theorem}[section]

\newtheorem{prop}[thm]{Proposition}
\newtheorem{lem}[thm]{Lemma}
\newtheorem{cor}[thm]{Corollary}

\theoremstyle{definition}
\newtheorem{dfn}[thm]{Definition}
\newtheorem{dfns}[thm]{Definitions}

\newtheorem{exmp}[thm]{Example}
\newtheorem{exmps}[thm]{Examples}
\newtheorem{rem}[thm]{Remark}

\newtheorem{dfns-rems}[thm]{Definitions and Remarks}
\newtheorem{notas-rems}[thm]{Notations and Remarks}
\newtheorem{exmps-rems}[thm]{Examples and Remarks}

\begin{document}


\title[Two lower bounds for the Stanley depth]{Two lower bounds for the Stanley depth of monomial ideals}


\author{L. Katth\"an}

\address{L. Katth\"an, Universit\"at Osnabr\"uck, FB Mathematik/Informatik, 49069 Osnabr\"uck, Germany}
\email{lkatthaen@uos.de}

\author[S. A. Seyed Fakhari]{S. A. Seyed Fakhari}

\address{S. A. Seyed Fakhari, School of Mathematics, Institute for Research
in Fundamental Sciences (IPM), P.O. Box 19395-5746, Tehran, Iran.}

\email{fakhari@ipm.ir}

\urladdr{http://math.ipm.ac.ir/fakhari/}


\begin{abstract}
Let $J\varsubsetneq I$ be two monomial ideals of the polynomial ring $S=\mathbb{K}[x_1,\ldots,x_n]$. In this paper, we provide two lower bounds for the Stanley depth of $I/J$. On the one hand, we introduce the notion of lcm number of $I/J$, denoted by $l(I/J)$, and prove that the inequality ${\rm sdepth}(I/J)\geq n-l(I/J)+1$ hold.
On the other hand, we show that ${\sdepth}(I/J)\geq n-\dim L_{I/J}$, where $\dim L_{I/J}$ denotes the order dimension of the lcm lattice of $I/J$.
We show that $I$ and $S/I$ satisfy Stanley's conjecture, if either the lcm number of $I$ or the order dimension of the lcm lattice of $I$ is small enough. Among other results, we also prove that the Stanley--Reisner ideal of a vertex decomposable simplicial complex satisfies Stanley's conjecture.
\end{abstract}






\thanks{The first author was partially supported by the German Research Council DFG-GRK~1916. The second author was in part supported by a grant from IPM (No. 93130422)}


\maketitle


\section{Introduction} \label{sec1}

Let $\mathbb{K}$ be a field and $S=\mathbb{K}[x_1,\dots,x_n]$ be the
polynomial ring in $n$ variables over the field $\mathbb{K}$. Let $M$ be a nonzero
finitely generated $\mathbb{Z}^n$-graded $S$-module. Let $u\in M$ be a
homogeneous element and $Z\subseteq \{x_1,\dots,x_n\}$. The $\mathbb
{K}$-subspace $u\mathbb{K}[Z]$ generated by all elements $uv$ with $v\in
\mathbb{K}[Z]$ is called a {\it Stanley space} of dimension $|Z|$, if it is
a free $\mathbb{K}[\mathbb{Z}]$-module. Here, as usual, $|Z|$ denotes the
number of elements of $Z$. A decomposition $\mathcal{D}$ of $M$ as a finite
direct sum of Stanley spaces is called a {\it Stanley decomposition} of
$M$. The minimum dimension of a Stanley space in $\mathcal{D}$ is called
{\it Stanley depth} of $\mathcal{D}$ and is denoted by ${\rm
sdepth}(\mathcal {D})$. The quantity $${\rm sdepth}(M):=\max\big\{{\rm
sdepth}(\mathcal{D})\mid \mathcal{D}\ {\rm is\ a\ Stanley\ decomposition\
of}\ M\big\}$$ is called {\it Stanley depth} of $M$. Stanley \cite{s}
conjectured that $${\rm depth}(M) \leq {\rm sdepth}(M)$$ for all
$\mathbb{Z}^n$-graded $S$-modules $M$. For a reader friendly introduction
to Stanley decomposition, we refer to \cite{psty} and for a nice survey on this topic we refer to \cite{h}. In this paper we prove Stanley's conjecture for some classes of monomial ideals.

Before stating the main results of this paper, we mention that for the monomials $u_1, \ldots, u_k \in S$, we denote their least common multiple by ${\rm lcm}(u_1, u_2, \ldots, u_k)$. Also, for a monomial ideal $I$, we denote by $G(I)$ the set of minimal monomial generators of $I$.

\begin{dfn}
Let $J\varsubsetneq I\subseteq S=\mathbb{K}[x_1,\ldots,x_n]$ be monomial ideals. The {\it lcm number} of $I/J$, denoted by $l(I/J)$ is the maximum integer $t$, for which there exist monomials $u_1, \ldots, u_t\in G(I)\cup G(J)$ such that $$u_1\neq {\rm lcm}(u_1, u_2)\neq \ldots \neq {\rm lcm}(u_1, u_2, \ldots, u_t).$$
\end{dfn}

\begin{rem}
We mention that the lcm number of monomial ideals was first considered by Terai to determine an upper bound for the arithmetical rank of squarefree monomial ideals (see \cite[Corollary 4]{k}).
\end{rem}

Let $J\varsubsetneq I$ be two monomial ideals. In Section \ref{sec2}, we determine lower bounds for the Stanley depth of $I/J$. More explicit, we prove that ${\rm sdepth}(I/J)\geq n-l(I/J)+1$. This, in particular, implies that $${\rm sdepth}(I)\geq n-l(I) +1 \ \ \  {\rm and} \ \ \  {\rm sdepth}(S/I)\geq n-l(I).$$

\begin{dfn}
Let $J\varsubsetneq I\subseteq S=\mathbb{K}[x_1,\ldots,x_n]$ be monomial ideals. The {\it lcm lattice} of $I/J$, denoted by $L_{I/J}$, is the set of all least common multiples of non-empty subsets of $G(I)\cup G(J)$, ordered by divisibility and augmented with an additional minimal element $\hat{0}$.
Moreover, we set $L_I := L_{I/0}$.
\end{dfn}
\begin{rem}
The lcm-lattice of a monomial ideal was introduced by Gasharov, Peeva and Welker in \cite{GPW}.
Note that the lcm number $l(I/J)$ is \emph{length} of $L_{I/J}$, i.e. one less than the maximal number of elements of a maximal chain in $L_{I/J}$.
\end{rem}
\begin{dfns}
\begin{enumerate}
\item Let $P$ and $P'$ be finite posets.
An \emph{embedding} is a map $j: P \rightarrow P'$ between two posets such that $p \leq q$ if and only if $j(p) \leq j(q)$ for $p,q\in P$.
\item The \emph{order dimension} of a poset, $\dim P$, is the minimal $d \in \NN$, such that there exists an embedding $P \rightarrow \NN^d$.
\end{enumerate}
\end{dfns}
Note that an embedding is necessarily injective and monotonic. Even if $P$ and $P'$ are lattices we do not require an embedding to respect the join.
We refer the reader to \cite{t} for background information about the dimension of posets.

Let $J\varsubsetneq I$ be two monomial ideals. In Section \ref{secDim}, we give a lower bound for the Stanley depth of $I/J$. Namely, we prove that ${\rm sdepth}(I/J)\geq n-\dim L_{I/J}$ and that ${\rm sdepth}(I)\geq n-\dim L_{I}+1$.

\begin{rem}
Both lower bounds for the Stanley depth are known to be bounds for the usual depth, in the case $I =S$. Indeed,
let $J \subset S$ be a monomial ideal.
By \cite{GPW} the projective dimension of $S/J$ can be computed from the homology of the order complex of lower intervals in $L_J$.
It is easy to see that the dimension of these order complexes is bound above by the lcm number $l(S/J)$. Hence \cite[Theorem 2.1]{GPW} implies that
\[ {\rm depth}(S/J)\geq n-l(S/J) \]
Moreover, it follows easily from Theorem 1 of \cite{rw} that
\[ {\rm depth}(S/J)\geq n-\dim L_J. \]
We provide proofs of both bounds on the general case below.
\end{rem}

In Section \ref{sec3}, we show that the Stanley--Reisner ideal of a vertex decomposable simplicial complex satisfies Stanley's conjecture (see Theorem \ref{decom}). Using this result and the above inequalities, we prove that $I$ and $S/I$ satisfy Stanley's conjecture provided that $l(I)\leq 3$ or $\dim L_I \leq 3$. (see Theorem \ref{degconj}).


\section{A lower bound for the Stanley depth} \label{sec2}

Let $J\varsubsetneq I$ be two monomial ideals. In this section, we prove the first main result of this paper. Indeed, in Theorem \ref{main}, we determine a lower bound for the Stanley depth of $I/J$. In \cite{s2}, the author provides linear algebraic lower bonds for the Stanley depth of $I$ and the Stanley depth of $S/I$, where $I$ is squarefree monomial ideal. The bound which will be proven in Theorem \ref{main} is stronger than these mentioned lower bounds, given in \cite{s2}. On the other hand, we do not focus on squarefree monomial ideals and consider a general monomial ideal.

To prove the main result, we need a couple of lemmas. The following lemma shows that the lcm number of a monomial ideal does not increase under the colon operation with respect to an arbitrary variable.

\begin{lem} \label{colon}
Let $J\varsubsetneq I$ be two monomial ideals of $S$. Then for every $1\leq i \leq n$, we have $l((I:x_i)/(J:x_i))\leq l(I/J)$.
\end{lem}
\begin{proof}
Assume without loss of generality that $i=1$. We note that $$(I:x_1)=\langle\frac{u}{{\rm gcd}(u, x_1)}: u\in G(I)\rangle$$and$$(J:x_1)=\langle\frac{u}{{\rm gcd}(u, x_1)}: u\in G(J)\rangle,$$where ${\rm gcd}(u, x_1)$ denotes the greatest common divisor of $u$ and $x_1$. Set $l((I:x_i)/(J:x_i))=t$ and suppose that $v_1, \ldots, v_t$ are monomials in $G((I:x_1))\cup G((J:x_1))$ such that $$v_1\neq {\rm lcm}(v_1, v_2)\neq \ldots \neq {\rm lcm}(v_1, v_2, \ldots, v_t).$$For every $1\leq j\leq t$, set $u_j=v_j$, if $v_j\in G(I)\cup G(J)$ and $u_j=x_1v_j$ if $v_j\notin G(I)\cup G(J)$. It is clear that in both cases $u_j\in G(I)\cup G(J)$. We claim that for every $1\leq k\leq t-1$, $${\rm lcm}(u_1, u_2, \ldots, u_k)\neq {\rm lcm}(u_1, u_2, \ldots, u_{k+1}).$$Indeed, if $v_j\in G(I)\cup G(J)$, for every $1\leq j\leq k$, then $v_j=u_j$, for every $1\leq j\leq k$ and thus $${\rm lcm}(u_1, u_2, \ldots, u_k)={\rm lcm}(v_1, v_2, \ldots, v_k)\neq {\rm lcm}(v_1, v_2, \ldots, v_{k+1})$$and since ${\rm lcm}(v_1, v_2, \ldots, v_{k+1})$ divides ${\rm lcm}(u_1, u_2, \ldots, u_{k+1})$, it follows that $${\rm lcm}(u_1, u_2, \ldots, u_k)\neq {\rm lcm}(u_1, u_2, \ldots, u_{k+1}).$$Now assume that $v_j\notin G(I)\cup G(J)$, for some $1\leq j\leq k$. Then $u_j=x_1v_j$ and hence $${\rm lcm}(u_1, u_2, \ldots, u_k)=x_1{\rm lcm}(v_1, v_2, \ldots, v_k)\neq$$ $$x_1{\rm lcm}(v_1, v_2, \ldots, v_{k+1})={\rm lcm}(u_1, u_2, \ldots, u_{k+1}).$$This proves the claim and shows that $$u_1\neq {\rm lcm}(u_1, u_2)\neq \ldots \neq {\rm lcm}(u_1, u_2, \ldots, u_t).$$Therefore, $l(I)\geq t$.
\end{proof}

In the following lemma, we consider the behavior of the lcm number of monomial ideals under the elimination of a variable. As usual, for every monomial $u$, the {\it support} of $u$, denoted by ${\rm Supp}(u)$, is the set of variables which divide $u$.

\begin{lem} \label{del}
Let $J\varsubsetneq I$ be two monomial ideals of $S=\mathbb{K}[x_1,\ldots,x_n]$, such that $$x_1\in \bigcup_{u\in G(I)\cup G(J)}{\rm Supp}(u).$$Let $S'=\mathbb{K}[x_2, \ldots, x_n]$ be the polynomial ring obtained from $S$ by deleting the variable $x_1$ and consider the ideals $I'=I\cap S'$ and $J'=J\cap S'$. Then $l(I'/J')+1\leq l(I/J)$.
\end{lem}
\begin{proof}
Assume that $l(I'/J')=t$. Suppose that $u_1, \ldots, u_t$ are monomials in $G(I')\cup G(J')$ such that $$u_1\neq {\rm lcm}(u_1, u_2)\neq \ldots \neq {\rm lcm}(u_1, u_2, \ldots, u_t).$$It is obvious that $u_j\in G(I)\cup G(J)$, for every $1\leq j \leq t$. By assumption, there exists a monomial, say $u_{t+1}\in G(I)\cup G(J)$, such that $x_1$ divides $u_{t+1}$. Since $u_1, \ldots, u_t$  do not divide $x_1$, it follows that for every $1\leq j \leq t$, $u_{t+1}\neq u_j$ and $$u_1\neq {\rm lcm}(u_1, u_2)\neq \ldots \neq {\rm lcm}(u_1, u_2, \ldots, u_t)\neq {\rm lcm}(u_1, u_2 \ldots, u_t, u_{t+1}).$$This shows that $l(I/J)\geq t+1$.
\end{proof}

\begin{rem}
It is completely clear from the proof of the Lemma \ref{del}, that one can consider any arbitrary variable instead of $x_1$.
\end{rem}

In the following theorem we determine a lower bound for the Stanley depth of $I/J$. We believe this bound is known to be a lower bound for depth also. But we did not find a reference and hence for the sake of completeness we provide a proof.

\begin{thm} \label{main}
Let $J\varsubsetneq I$ be two monomial ideals of $S=\mathbb{K}[x_1,\ldots,x_n]$. Then ${\rm depth}(I/J)\geq n-l(I/J)+1$ and ${\rm sdepth}(I/J)\geq n-l(I/J)+1$.
\end{thm}
\begin{proof}
We prove the assertions  by induction on $n$ and $$\sum_{u\in G(I)\cup G(J)} {\rm deg} (u).$$ The assertion can be checked easily, when
$n=1$ or $$\sum_{u\in G(I)\cup G(J)} {\rm deg} (u)=1.$$

We now assume that $n\geq 2$ and
$$\sum_{u\in G(I)\cup G(J)} {\rm deg} (u)\geq 2.$$Let $S'=\mathbb{K}[x_2, \ldots, x_n]$ be the polynomial ring obtained from $S$ by deleting the variable $x_1$ and consider the ideals $I'=I\cap S'$, $J'=J\cap S'$, $I''=(I:x_1)$ and
$J''=(J:x_1)$. If $$x_1\notin \bigcup_{u\in G(I)\cup G(J)}{\rm Supp}(u),$$then trivially ${\rm depth}(I/J)={\rm
depth}_{S'}(I'/J')+1$ and by \cite[Lemma 3.6]{hvz}, we conclude that ${\rm sdepth}(I/J)={\rm
sdepth}_{S'}(I'/J')+1$.
On the other hand it is clear that $l(I/J)=l(I'/J')$.
Therefore, using the induction hypothesis on $n$ we conclude that ${\rm depth}(I/J)\geq n-l(I/J)+1$ and ${\rm sdepth}(I/J)\geq n-l(I/J)+1$.
Therefore, we may assume that $$x_1\in \bigcup_{u\in G(I)\cup G(J)}{\rm Supp}(u).$$

Now $I/J=(I'S'/J'S')\oplus x_1(I''S/J''S)$ and therefore by definition of  the Stanley depth we have
\[
\begin{array}{rl}
{\rm sdepth}(I/J)\geq \min \{{\rm sdepth}_{S'}(I'S'/J'S), {\rm sdepth}_S(I''/J'')\},
\end{array} \tag{1} \label{1}
\]

On the other hand, by applying the depth lemma on the exact sequence
\[
\begin{array}{rl}
0\longrightarrow I''/J''\stackrel{.x_1}{\longrightarrow} I/J\longrightarrow I/(x_1I''+J)
\longrightarrow 0
\end{array}
\]
we conclude that
\[
\begin{array}{rl}
{\rm depth}(I/J)\geq \min \{{\rm depth}_S(I''/J''), {\rm depth}_S(I/(x_1I''+J)).
\end{array} \tag{2} \label{2}
\]

We note that every $I'S'/J'S'$-regular sequence in $S'$ is also a regular sequence for $I/(x_1I''+J)$. This shows that ${\rm depth}_S(I/(x_1I''+J))\geq  {\rm sdepth}_{S'}(I'S'/J'S)$. Hence it follows from inequality (\ref{2}) that
\[
\begin{array}{rl}
{\rm sdepth}(I/J)\geq \min \{{\rm sdepth}_{S'}(I'S'/J'S), {\rm sdepth}_S(I''/J'')\},
\end{array} \tag{3} \label{3}
\]
Using Lemma \ref{colon} we conclude that that $l(I''/J'')\leq l(I/J)$. Hence our induction hypothesis on $$\sum_{u\in G(I)\cup G(J)} {\rm deg}(u)$$
implies that $${\rm sdepth}_S(I''/J'')\geq n-l(I''/J'')+1\geq n-l(I)+1$$and similarly ${\rm depth}_S(I''/J'')\geq n-l(I)+1$.

On the other hand, since $$x_1 \in \bigcup_{u\in G(I)\cup G(J)} {\rm Supp}(u),$$using Lemma \ref{del} we conclude that $l(I'S'/J'S)\leq l(I/J)-1$ and therefore by
the induction hypothesis on $n$ we conclude that

$${\rm sdepth}_{S'}(I'S'/J'S)\geq (n-1)-l(I'S'/J'S)+1\geq (n-1)-(l(I/J)-1)+1$$ $$=n-l(I/J)+1$$ and similarly ${\rm depth}_{S'}(I'S'/J'S)\geq n-l(I/J)+1$.
Now the assertions follow from inequalities (\ref{1}) and (\ref{3}).
\end{proof}

As an immediate consequence of Theorem \ref{main} we obtain the following result.

\begin{cor} \label{lcm}
Let $I$ be a monomial ideal of $S=\mathbb{K}[x_1,\ldots,x_n]$. Then ${\rm sdepth}(I)\geq n-l(I)+1$ and ${\rm sdepth}(S/I)\geq n-l(I)$.
\end{cor}

For every vector $\mathbf{a}=(a_1, \ldots, a_n)$ of non-negative integers, we denote the monomial $x_1^{a_1}\ldots x_n^{a_n}$ by $\mathbf{x^{a}}$. Let $I\subseteq S$ be a monomial ideal and $G(I)=\{\mathbf{x^{a_1}}, \ldots, \mathbf{x^{a_m}}\}$ be the set of minimal monomial generators of $I$. The {\it rank} of $I$, denoted by ${\rm rank}(I)$ is the cardinality of the largest $\mathbb{Q}$-linearly independent subset of $\{\mathbf{a_1}, \ldots, \mathbf{a_m}\}$, where $\mathbb{Q}$ is the set of rational numbers. In \cite{s2}, the author proves that for every squarefree monomial ideal of $S$ the inequalities ${\rm sdepth}(I)\geq n-{\rm rank}(I)+1$ and ${\rm sdepth}(S/I)\geq n-{\rm rank}(I)$ hold. We note that Corollary \ref{lcm} implies this result.

\begin{cor} \label{rank}
{\rm (}\cite[Thorem 3.3]{s2}{\rm )} Let $I$ be a squarefree monomial ideal of $S=\mathbb{K}[x_1,\ldots,x_n]$. Then ${\rm sdepth}(I)\geq n-{\rm rank}(I)+1$ and ${\rm sdepth}(S/I)\geq n-{\rm rank}(I)$.
\end{cor}
\begin{proof}
Assume that $l(I)=t$. Suppose that $u_1, \ldots, u_t$ are monomials in the set of minimal monomial generators of $I$ such that
\[
\begin{array}{rl}
u_1\neq {\rm lcm}(u_1, u_2)\neq \ldots \neq {\rm lcm}(u_1, u_2, \ldots, u_t).
\end{array} \tag{$\ast$} \label{ast}
\]
Since $u_1, \ldots, u_t$ are squarefree, inequalities (\ref{ast}) imply that $${\rm Supp}(u_1)\subsetneqq \bigcup_{i=1}^2{\rm Supp}(u_i)\subsetneqq \ldots \subsetneqq \bigcup_{i=1}^t{\rm Supp}(u_i).$$ This shows that $u_1, \ldots, u_t$ are $\mathbb{Q}$-linearly independent and thus ${\rm rank}(I)\geq t$. Now Corollary \ref{lcm} completes the proof.
\end{proof}

Let $I$ be a monomial ideal and assume that $G(I)$ is the set of minimal monomial generators of $I$. The {\it initial degree} of $I$, denote by ${\rm indeg}(I)$ is the minimum degree of the monomials belonging to $G(I)$. The following proposition provides an upper bound for the lcm number of a squarefree monomial ideal in terms of its initial degree.

\begin{prop} \label{lcmindeg}
Let $I$ be a squarefree monomial ideal of $S=\mathbb{K}[x_1,\ldots,x_n]$. Then $l(I)\leq n-{\rm indeg}(I)+1$.
\end{prop}
\begin{proof}
Similar to the proof of Corollary \ref{rank}, one can see that if $u_1, \ldots, u_t$ are monomials in the set of minimal monomial generators of $I$ such that $$u_1\neq {\rm lcm}(u_1, u_2)\neq \ldots \neq {\rm lcm}(u_1, u_2, \ldots, u_t),$$then
\[
\begin{array}{rl}
{\rm Supp}(u_1)\subsetneqq \bigcup_{i=1}^2{\rm Supp}(u_i)\subsetneqq \ldots \subsetneqq \bigcup_{i=1}^t{\rm Supp}(u_i).
\end{array} \tag{$\ast\ast$} \label{astast}
\]
Since $u_1, \ldots, u_t$ are squarefree monomials, the cardinality of ${\rm Supp}(u_1)$ is greater than or equal to ${\rm indeg}(I)$. On the other hand, the cardinality of $\bigcup_{i=1}^t{\rm Supp}(u_i)$ is at most $n$. Hence, the inclusions (\ref{astast}) show that
$$l(I)\leq \mid\bigcup_{i=1}^t{\rm Supp}(u_i)\mid-\mid{\rm Supp}(u_1)\mid+1\leq n-{\rm indeg}(I)+1.$$
\end{proof}

The following corollary is an immediate consequence of Corollary \ref{lcm} and Corollary \ref{lcmindeg}.

\begin{cor} \label{indeg}
Let $I$ be a squarefree monomial ideal of $S=\mathbb{K}[x_1,\ldots,x_n]$. Then ${\rm sdepth}(I)\geq {\rm indeg}(I)$ and ${\rm sdepth}(S/I)\geq {\rm indeg}(I)-1$.
\end{cor}

\begin{rem}
Let $I$ be a squarefree monomial ideal. We mention that the inequality ${\rm sdepth}(I)\geq {\rm indeg}(I)$ was known by \cite[Proposition 3.1]{hvz}.
\end{rem}

Let $I$ be a monomial ideal of $S=\mathbb{K}[x_1,\ldots,x_n]$ and $G(I)$ be the set of minimal monomial generators of $I$. Assume that $\mid G(I)\mid=m$. Cimpoea{\c{s}} \cite{c1} proves that ${\rm sdepth}(S/I)\geq n-m$ (see \cite[Proposition 1.2]{c1}). It is completely clear that the bound given in Corollary \ref{lcm} for the Stanley depth of $S/I$ is better than the bound given by Cimpoea{\c{s}}. Indeed, there are examples (see Example \ref{examp}), for which $m-l(I)$ is large enough and the inequality ${\rm sdepth}(S/I)\geq n-l(I)$ is sharp for them, i.e., ${\rm sdepth}(S/I)= n-l(I)$.

\begin{exmp} \label{examp}
Let$$I=(x_ix_j : 1\leq i< j\leq n)$$ be a monomial ideal of $S=\mathbb{K}[x_1, \ldots, x_n]$. Then $l(I)=n-1$ and thus$$m-l(I)=\frac{(n-1)(n-2)}{2},$$ where $m$ is the cardinality of the set of minimal monomial generators of $I$. This shows that by choosing a suitable $n$, the number $m-l(I)$ can be larger than any given integer. On the other hand, the height of every associated prime of $I$ is equal to $n-1$. Thus, it follows from \cite[Proposition 1.3]{hvz} and Corollary \ref{lcm} that ${\rm sdepth}(S/I)=1=n-l(I)$.
\end{exmp}


\section{Stanley depth and order dimension}\label{secDim}
In this section, we give the proof of our second main result.
Let us recall some definitions of lattice theory. For a comprehensive treatment of this subject we refer the reader to \cite{g}.
Recall that a join-semilattice is a poset in which every two elements have a least upper bound, called their join.
We call a subset $L'$ of a finite join-semilattice $L$ a \emph{join-subsemilattice} if it is a join-semilattice with the induced join-operation from $L$.
It is well-known that every finite join-semilattice with a minimal element is in fact a lattice. However, as we will never consider the meet, it is more convenient to work in the category of join-semilattices.
An element $m \in L$ is called \emph{join-irreducible} if it cannot be written as the join of two elements different from $m$.
Note that every element $m$ in a finite join-semilattice is the join of the set of all join-irreducible elements less than or equal to $m$.

The following is a convenient characterization of the dimension of a finite join-semilattice.
\newcommand{\pinv}[1]{{#1}^{\dagger}}
\begin{lem}
	Let $L$ be a finite join-semilattice and let $d \in \NN$. Then the following are equivalent:
	\begin{enumerate}
	\item There exists a surjective join-preserving map $\phi: L' \rightarrow L$ for a finite join-subsemilattice $L'$ of $\NN^d$.
	\item $\dim L \leq d$.
	\end{enumerate}
\end{lem}
\begin{proof}
1) $\Rightarrow$ 2): Consider the map $\pinv{\phi}: L \rightarrow \NN^d$ defined by $\pinv{\phi}(a) := \bigvee \phi^{-1}(a)$. It is an embedding of $L$ into $\NN^d$ by \cite[Lemma 4.1]{ikm2}, hence $\dim L \leq d$.

2) $\Rightarrow$ 1): Let $j: L \rightarrow \NN^d$ be an embedding and set $L' \subset \NN^d$ to be the join-subsemilattice of $\NN^d$ generated by the image of $j$, i.e. the set of all joins of subsets of the image of $L$. 
Define $\phi: L' \rightarrow L$ by $\phi(x') := \bigvee\set{x \in L \with j(x) \leq x'}$.
This map is clearly monotonic.
Moreover, monotonicity implies that $j(x \vee y) \geq j(x) \vee j(y)$ and $\phi(x' \vee y') \geq \phi(x') \vee \phi(y')$.
So it remains to show that $\phi$ is surjective and that $\phi(x' \vee y') \leq \phi(x') \vee \phi(y')$ for $x', y' \in L'$.

For the first claim, we show that $\phi\circ j = Id_L$. Indeed,
\[ \phi(j(x)) = \bigvee\set{y \in L \with j(y) \leq j(x)} = \bigvee\set{y \in L \with y \leq x} = x \]
for every $x \in L$.
Moreover, we claim that for every $x' \in L'$ we have
\[
\begin{array}{rl}
j(\phi(x')) \geq x'.
\end{array} \tag{$\dag$} \label{dag}
\]
To see this we compute
\begin{align*}
j(\phi(x')) &= j(\bigvee\set{y \in L \with j(y) \leq x'}) \geq \bigvee\set{j(y) \with y \in L, j(y) \leq x'} \\
&\geq \bigvee\set{y' \in L' \with y' \text{ join-irreducible}, y' \leq x'} = x'
\end{align*}
Here we used that $j$ is surjective onto the join-irreducible elements of $L'$.
Now \ref{dag} implies that $$x' \vee y' \leq x' \vee j(\phi(y')) \leq j(\phi(x')) \vee j(\phi(y'))$$and thus
\[ \phi(x' \vee y') \leq \phi(j(\phi(x')) \vee j(\phi(y'))) \leq \phi(j(\phi(x') \vee \phi(y'))) = \phi(x') \vee \phi(y'). \]
\end{proof}

\begin{thm}\label{thm:dim}
	Let $J \subsetneq I \subset S$ be two monomial ideals. Then
	\[ \sdepth_S (I/J) \geq n - \dim L_{I/J} \]
	and
	 \[\sdepth_S (I) \geq n - \dim L_I + 1. \]
\end{thm}
\begin{proof}
	Let $d := \dim L_{I/J}$. Consider the join-semilattice $L' \subset \NN^d$ of the preceding lemma and the corresponding surjective join-preserving map $\phi: L' \rightarrow L_{I/J}$.
	Let moreover $L'' \subset L'$ be the join-subsemilattice corresponding to $L_J$, i.e. the join-{subsemilattice} generated by the images of $L_J$ in $\NN^d$.
	By construction $\phi$ maps $L''$ onto $L_J$.
	We interpret the elements of $L'$ as exponent vectors to see that $L'$ and $L''$ are lcm lattices of $I'/J'$ and $J'$, for two monomial ideals $J' \subsetneq I' \subset S' = \KK[x_1,\dotsc, x_d]$ in $d$ variables.
	Now it follows from \cite[Theorem 4.9]{ikm2} that $n - \sdepth_S (I/J) \leq d - \sdepth_{S'} (I'/J') \leq d$ and thus
	\[ \sdepth_S (I/J) \geq n - d. \]
	
	Moreover, by the same argument $n - \sdepth_S(I) \leq d - \sdepth_{S'} (I') \leq d - 1$ and hence \[ \sdepth_S (I) \geq n - d + 1. \]
\end{proof}

\begin{rem}
	It also holds that
	\[ \depth_S (I/J) \geq n - \dim L_{I/J} \]
	and
	\[\depth_S (I) \geq n - \dim L_I + 1. \]
	This is proven by the same argument, using \cite[Theorem 4.11]{ikm2} instead of \cite[Theorem 4.9]{ikm2}.
\end{rem}

We present two examples to show that in general there is no inequality between the lcm number of $I$ and the order dimension of $L_I$.
\begin{exmps}
\begin{enumerate}
\item Consider the ideal $I = (x^2, xy, y^2) \subset S = \KK[x,y,z]$.
	It is easy to see that $l(I) = 3$, so Theorem \ref{lcm} gives the bound $\sdepth_S S/I \geq 3 - 3 = 0$. On the other hand, $\dim L_I = 2$ (the exponent vectors give an embedding into $\NN^2$), so Theorem \ref{thm:dim} gives the better bound $\sdepth_S S/I \geq 3 - 2 = 1$.
\item Let $I \subset S = \KK[x_1, \dotsc, x_5]$ be the ideal generated by all squarefree monomials of degree $3$.
	Again, we have that $l(I) = 3$, so Theorem \ref{lcm} gives the bound $\sdepth_S S/I \geq 5 - 3 = 2$. We computationally verified that $\dim L_I = 4$, so in this case Theorem \ref{thm:dim} gives the worse bound $\sdepth_S S/I \geq 5 - 4 = 1$.
\end{enumerate}
\end{exmps}


\section{Monomial ideals with small lcm number and order dimension} \label{sec3}

In this section, we prove that the Stanley--Reisner ideal of a vertex decomposable simplicial complex satisfies Stanley's Conjecture (see Theorem \ref{decom}). Using this result, Corollary \ref{lcm} and Theorem \ref{thm:dim}, we prove that $I$ and $S/I$ satisfy Stanley's conjecture if
\begin{itemize}
\item[(i)]  $l(I)\leq 3$ or
\item[(ii)]  $\dim L_I \leq 3$ or
\item[(iii)] $l(I)\leq 4$ and $S/I$ is Gorenstein or
\item[(iv)] $\dim L_I \leq 4$ and $S/I$ is Gorenstein.
\end{itemize}

To state and prove the next results, we need to introduce some notation and well-known facts from combinatorial commutative algebra.

A {\it simplicial complex} $\Delta$ on the set of vertices $[n]:=\{1,
\ldots,n\}$ is a collection of subsets of $[n]$ which is closed under
taking subsets; that is, if $F \in \Delta$ and $F'\subseteq F$, then also
$F'\in\Delta$. Every element $F\in\Delta$ is called a {\it face} of
$\Delta$, the {\it size} of a face $F$ is defined to be $|F|$ and its {\it
dimension} is defined to be $|F|-1$. (As usual, for a given finite set $X$,
the number of elements of $X$ is denoted by $|X|$.) The {\it dimension} of
$\Delta$ which is denoted by $\dim\Delta$, is defined to be $d-1$, where $d
=\max\{|F|\mid F\in\Delta\}$. A {\it facet} of $\Delta$ is a maximal face
of $\Delta$ with respect to inclusion. Let $\mathcal{F}(\Delta)$ denote the
set of facets of $\Delta$. It is clear that $\mathcal{F}(\Delta)$
determines $\Delta$. When $\mathcal{F}(\Delta)= \{F_1,\ldots,F_m\}$, we
write $\Delta=\langle F_1, \ldots,F_m\rangle$. We say that $\Delta$ is {\it
pure} if all facets of $\Delta$ have the same cardinality. The {\it link}
of $\Delta$ with respect to a face $F \in \Delta$, denoted by ${\rm
lk_{\Delta}}(F)$, is the simplicial complex ${\rm lk_{\Delta}}(F)=\{G
\subseteq [n]\setminus F\mid G\cup F\in \Delta\}$ and the {\it deletion} of
$F$, denoted by ${\rm del_{\Delta}}(F)$, is the simplicial complex ${\rm
del_{\Delta}}(F)=\{G \subseteq [n]\setminus F\mid G \in \Delta\}$. When $F
= \{x\}$ is a single vertex, we abuse notation and write ${\rm lk_{\Delta
}}(x)$ and ${\rm del_{\Delta}}(x)$.

Let $S=\mathbb{K}[x_1,\ldots, x_n]$ and let $\Delta$ be a simplicial complex on $[n]$. For every
subset $F\subseteq [n]$, we set ${\it x}_F=\prod_{i\in F}x_i$. The {\it
Stanley--Reisner ideal of $\Delta$ over $\mathbb{K}$} is the ideal $I_{
\Delta}$ of $S$ which is generated by those squarefree monomials $x_F$ with
$F\notin\Delta$. In other words, $I_{\Delta}=\langle{\it x}_F\mid
F\in\mathcal{N}(\Delta) \rangle$, where $\mathcal{N}(\Delta)$ denotes the
set of minimal nonfaces of $\Delta$ with respect to inclusion. The {\it
Stanley--Reisner ring of $\Delta$ over $\mathbb{K}$}, denoted by $\mathbb
{K}[\Delta]$, is defined to be $\mathbb{K}[\Delta]=S/I_{\Delta}$. Let
$I\subseteq S$ be an arbitrary squarefree monomial ideal. Then there is a
unique simplicial complex $\Delta$ such that $I = I_{\Delta}$. A simplicial complex $\Delta$ is said to be {\it Cohen--Macaulay} if $\mathbb{K}[\Delta]$ is Cohen--Macaulay. For every integer $0\leq i\leq {\rm dim}\Delta$
the simplicial complex $\Delta^{[i]}:=\langle
F\in\Delta\mid\dim F=i\rangle$ is called the $i$-{\it pure skeleton} of
$\Delta$. A simplicial complex $\Delta$ is said to be {\it sequentially Cohen--Macaulay} if $\Delta^{[i]}$ is Cohen--Macaulay, for every $0\leq i\leq {\rm dim}\Delta$.

Let $\Delta$ be a simplicial complex on the vertex set $[n]$. Then we say that $\Delta$ is {\it vertex decomposable} if either
\begin{itemize}
\item[(1)] $\Delta$ is a simplex, i.e., a simplicial complex with only one facet or\\[-0.3cm]

\item[(2)] there exists $k\in [n]$ such that ${\rm del_{\Delta}}(k)$ and
${\rm lk_{\Delta}}(k)$ are vertex decomposable and every facet of
${\rm del_{\Delta}}(k)$ is a facet of $\Delta$.
\end{itemize}

It is know that every vertex decomposable simplicial complex is sequentially Cohen--Macaulay (see for example \cite{w}).

\begin{rem} \label{depth}
It follows from \cite[Corollary 3.33]{mv} that for every vertex decomposable simplical complex $\Delta$, we have
$${\rm depth}(\mathbb{K}[\Delta])={\rm min}\big\{\mid F\mid : F {\rm \ is \ a \ facet \ of \ } \Delta\big\}.$$
\end{rem}

In the next theorem, we prove that the Stanley--Reisner ideal of vertex decomposable simplicial complexes satisfy Stanley's conjecture.

\begin{thm} \label{decom}
Let $\Delta$ be a vertex decomposable simplicial complex. Then $I_{\Delta}$ satisfies Stanley's conjecture.
\end{thm}
\begin{proof}
We prove the assertion by induction on $n$. If $\Delta$ is a simplex, then $I_{\Delta}=0$ and there is nothing to prove. Thus, assume that $\Delta$ is not a simplex. Therefore, there exists a vertex $k\in [n]$ such that ${\rm del_{\Delta}}(k)$ and
${\rm lk_{\Delta}}(k)$ are vertex decomposable and every facet of
${\rm del_{\Delta}}(k)$ is a facet of $\Delta$. Let $S'=\mathbb{K}[x_1, \ldots, x_{k-1}, x_{k+1}, \ldots, x_n]$ be the polynomial ring obtained from $S$ by deleting the variable $x_k$ and consider the ideals $I'=I\cap S'$ and
$I''=(I:x_k)$. Now $I=I'S'\oplus x_kI''S$ and therefore by definition of  the Stanley depth we have
\[
\begin{array}{rl}
{\rm sdepth}(I)\geq \min \{{\rm sdepth}_{S'}(I'S'), {\rm sdepth}_S(I'')\}.
\end{array} \tag{$\ddagger$} \label{ddag}
\]

Note that $I''=(I:x_k)$ is the Stanley--Reisner ideal of ${\rm lk_{\Delta}}(k)$, considered as a simplicial complex on $[n]\setminus\{k\}$. Since ${\rm lk_{\Delta}}(k)$ is a vertex decomposable simplicial complex, it follows from \cite[Lemma 3.6]{hvz}, Remark \ref{depth} and the induction hypothesis that
$${\rm sdepth}_S(I'')={\rm sdepth}_{S'}(I'')+1\geq {\rm depth}_{S'}(I'')+1$$
$$\geq({\rm depth}_S(I)-1)+1={\rm depth}_S(I).$$

On the other hand, $I'=I\cap S'$ is the Stanley--Reisner ideal of ${\rm del_{\Delta}}(k)$, considered as a simplicial complex on $[n]\setminus\{k\}$. Since ${\rm del_{\Delta}}(k)$ is a vertex decomposable simplicial complex, it follows from Remark \ref{depth} and the induction hypothesis that
$${\rm sdepth}_{S'}(I'S')\geq {\rm depth}_{S'}(I'S')\geq {\rm depth}_S(I).$$

Now inequality (\ref{ddag}) completes the proof.
\end{proof}

Let $I$ be a monomial ideal. In \cite{hsy}, the authors prove that $S/I$ satisfies Stanley's conjecture, provided that ${\rm depth}(S/I)\geq n-1$ (see \cite[Corollary 2.3]{hsy}). The following lemma is an extension of this result.

\begin{lem} \label{codim}
Let $I$ be a monomial ideal of $S=\mathbb{K}[x_1,\ldots,x_n]$ and assume that ${\rm depth}(S/I)\geq n-2$. Then $I$ and $S/I$ satisfy Stanley's conjecture.
\end{lem}
\begin{proof}
We use induction on $\sum_{u\in G(I)}{\rm deg}(u)$, where $G(I)$ is the set of minimal monomial generators of $I$. If $$\sum_{u\in G(I)}{\rm deg}(u)=1,$$then $I$ is a principal ideal. Therefore, it follows from \cite[Theorem
1.1]{r} that ${\rm sdepth}(S/I)=n-1$. On the other hand, it is clear that ${\rm sdepth}(I)=n$. Thus, $I$ and $S/I$ satisfy Stanley's conjecture. Now, we consider the following cases.

{\bf Case 1.} ${\rm dim}(S/I)={\rm depth}(S/I)= n-2$. In this case $S/I$ is Cohen--Macaulay and the height of $I$ is equal to $2$. Thus $S/I$ satisfies Stanley's conjecture by \cite[Proposition 2.4]{hsy}. To prove that $I$ satisfies Stanley's conjecture, let $I^p$ denote the polarization of $I$ which is considered in a new polynomial ring, say $T$ (see \cite{hh'} for the definition of polarization). Then by \cite[Corollary 1.6.3]{hh'}, we conclude that $T/I^p$ is Cohen--Macaulay and the height of $I^p$ is equal to $2$. It follows from \cite[Theorem 2.3]{as} (see also \cite{m}) that $I^p$ is the Stanley--Reisner ideal of a vertex decomposable simplicial complex and therefore, $I^p$ satisfies Stanley's conjecture by Theorem \ref{decom}. Now \cite[Corollary 4.5]{ikm} implies that $I$ satisfies Stanley's conjecture.\\

{\bf Case 2.} ${\rm dim}(S/I)={\rm depth}(S/I)= n-1$. In this case, the height of every associated prime of $I$ is equal to one. Thus, $I$ is a principal ideal. Therefore, it follows from \cite[Theorem
1.1]{r} that ${\rm sdepth}(S/I)=n-1$. On the other hand, it is clear that ${\rm sdepth}(I)=n$. Thus, $I$ and $S/I$ satisfy Stanley's conjecture.\\

{\bf Case 3.} ${\rm dim}(S/I)=n-1$ and ${\rm depth}(S/I)= n-2$. In this case, the height of at least one of the associated primes of $I$ is equal to one. Hence, there exists a variable, say $x_k$, such that $I\subset (x_k)$. Thus, $I=x_k(I:x_k)$. This shows that $I$ and $(I:x_k)$ are isomorphic (as $\mathbb{Z}^n$-graded $S$-module). Thus ${\rm depth}(I)={\rm depth}((I:x_k))$, which implies that ${\rm depth}(S/I)={\rm depth}(S/(I:x_k))$. On the other hand, it follows from \cite[Theorem 1.1]{c1} that ${\rm sdepth}(I)={\rm sdepth}((I:x_k))$ and ${\rm sdepth}(S/I)={\rm sdepth}(S/(I:x_k))$. Hence, the induction hypothesis implies $${\rm sdepth}(I)={\rm sdepth}((I:x_k))\geq {\rm depth}((I:x_k))={\rm depth}(I).$$Similarly, ${\rm sdepth}(S/I)\geq {\rm depth}(S/I)$. Therefore, $I$ and $S/I$ satisfy Stanley's conjecture.

\end{proof}

Let $I$ be a monomial ideal. In the following theorem, we prove that $I$ and $S/I$ satisfies Stanley's conjecture, if the lcm number of $I$ or the order diension of $L_I$ is small.

\begin{thm} \label{degconj}
Let $I$ be a monomial ideal of $S=\mathbb{K}[x_1,\ldots,x_n]$.
If $l(I)\leq 3$ or $\dim L_I \leq 3$ holds, then $I$ and $S/I$ satisfy Stanley's conjecture.
\end{thm}
\begin{proof}
It follows from Corollary \ref{lcm} resp. Theorem \ref{thm:dim} that ${\rm sdepth}(S/I)\geq n-3$ and ${\rm sdepth}(I)\geq n-2$. This implies that if ${\rm depth}(S/I)\leq n-3$, then $I$ and $S/I$ satisfy Stanley's conjecture. Otherwise, the assertions follow from Lemma \ref{codim}.
\end{proof}

In the following corollary, we consider the Gorenstein monomial ideals with lcm number or order dimension at most four.

\begin{cor} \label{gore}
Let $I$ be a monomial ideal of $S=\mathbb{K}[x_1,\ldots,x_n]$ such that $S/I$ is Gorenstein.
If $l(I)\leq 4$ or $\dim L_I \leq 4$, then $I$ and $S/I$ satisfy Stanley's conjecture.
\end{cor}
\begin{proof}
By Corollary \ref{lcm} resp. Theorem \ref{thm:dim}, we conclude that ${\rm sdepth}(S/I)\geq n-4$ and ${\rm sdepth}(I)\geq n-3$. Thus, there is nothing to prove, if ${\rm depth}(S/I)\leq n-4$. If ${\rm depth}(S/I)\geq n-2$, then the assertions follow from Lemma \ref{codim}. Thus, we assume that ${\rm depth}(S/I)=n-3$. In this case, the height of $I$ is equal to $3$ and it follows from \cite[Theorem 3.1]{hsy} that $S/I$ satisfies Stanley's conjecture. In order to prove that $I$ satisfies Stanley's conjecture, we use the machinery of Polarization. Let $I^p$ denote the polarization of $I$ which is considered in a new polynomial ring, say $T$. Then by \cite[Corollary 1.6.3]{hh'}, we conclude that $T/I^p$ is Gorenstein and the height of $I^p$ is equal to $3$. Using \cite[Theorem 2.5]{as}, we conclude that $I$
is the Stanley--Reisner ideal of a vertex decomposable simplicial complex and therefore, $I^p$ satisfies Stanley's conjecture by Theorem \ref{decom}. Now \cite[Corollary 4.5]{ikm} implies that $I$ satisfies Stanley's conjecture.
\end{proof}

As an immediate consequence of Proposition \ref{lcmindeg}, Theorem \ref{degconj} and Corollary \ref{gore} we obtain the following result.

\begin{cor} \label{sf}
Let $I$ be a squarefree monomial ideal of $S=\mathbb{K}[x_1,\ldots,x_n]$. Then $I$ and $S/I$ satisfy Stanley's conjecture if
\begin{itemize}
\item[(i)] ${\rm indeg}(I)\geq n-2$ or
\item[(ii)] $S/I$ is Gorenstein and ${\rm indeg}(I)\geq n-3$
\end{itemize}
\end{cor}

A simplicial complex $\Delta$ is called {\it doubly Cohen--Macaulay} if $\Delta$ is Cohen--Macaulay and for every vertex $x$ of $\Delta$, the simplicial complex ${\rm del_{\Delta}}(x)$ is Cohen--Macaulay of the same dimension as $\Delta$. The following corollary shows that $I_{\Delta}$ and $S/I_{\Delta}$ satisfy Stanley's conjecture if $\Delta$ is a doubly Cohen--Macaulay simplicial complex and the initial degree of $I_{\Delta}$ is large enough.

\begin{cor} \label{doubl}
Let $I$ be the Stanley--Reisner ideal of a doubly Cohen--Macaulay simplicial complex and assume that ${\rm indeg}(I_{\Delta})\geq n-3$. Then $I$ and $S/I$ satisfy Stanley's conjecture.
\end{cor}
\begin{proof}
By Corollary \ref{lcm}, we conclude that ${\rm sdepth}(S/I)\geq n-4$ and ${\rm sdepth}(I)\geq n-3$. Thus, the assertion is true, if ${\rm depth}(S/I)\leq n-4$. If ${\rm depth}(S/I)=n-3$, then the height of $I$ is equal to $3$. In this case, $S/I$ satisfies Stanley's conjecture by \cite[Theorem 4.2]{nr}. On the other hand, it follows from \cite[Theorem 2.13]{as} that $I$ is the Stanley--Reisner ideal of a vertex decomposable simplicial complex and therefore, $I$ satisfies Stanley's conjecture by Theorem \ref{decom}. The remaining case (${\rm depth}(S/I)\geq n-2$) follows from Lemma \ref{codim}.
\end{proof}


\section*{Acknowledgment}
The authors wish to thank Christian Bey for pointing out the reference \cite{rw} to us.



\end{document}